\documentclass[11pt]{amsart}
\usepackage{amsmath,amsfonts,latexsym}

\usepackage{palatino}                                
\usepackage[abs]{overpic}                            

\usepackage{xcolor}
\definecolor{indigo}{rgb}{0.29, 0.0, 0.51}
\usepackage[colorlinks, urlcolor=indigo, linkcolor=indigo, citecolor=indigo]{hyperref}

\usepackage[hcentering, top = 1.5in, total={6in, 8.3in}]{geometry}  

\newtheorem{question}{Question}
\newtheorem*{conjecture}{Conjecture}
\newtheorem{theorem}{Theorem}
\newtheorem{proposition}[theorem]{Proposition}
\newtheorem{lemma}[theorem]{Lemma}
\newtheorem{corollary}[theorem]{Corollary}

\theoremstyle{definition}
\newtheorem{definition}[theorem]{Definition}

\theoremstyle{remark}
\newtheorem{remark}[theorem]{Remark}
\newtheorem{example}[theorem]{Example}

\def\dfn#1{{\em #1}}
\def\R{\mathbb{R}}
\def\Z{\mathbb{Z}}

\def\F{\mathbb{F}}
\def\a{\mathbf{a}}
\def\b{\mathbf{b}}
\def\z{\mathbf{z}}

\def\alphas{\boldsymbol{\alpha}}
\def\betas{\boldsymbol{\beta}}
\def\omegas{\boldsymbol{\omega}}
\def\xs{\mathbf{x}}
\def\ys{\mathbf{y}}

\newcommand{\frS}{\mathfrak{S}}

\newcommand{\frs}{\mathfrak{s}}
\newcommand{\frt}{\mathfrak{t}}
\newcommand{\fru}{\mathfrak{u}}

\newcommand{\cC}{\mathcal{C}}
\newcommand{\cD}{\mathcal{D}}
\newcommand{\cH}{\mathcal{H}}

\newcommand{\cM}{\mathcal{M}}

\DeclareMathOperator{\Hom}{{Hom}}
\DeclareMathOperator{\Lef}{{Lef}}
\DeclareMathOperator{\Spin}{{Spin}}
\DeclareMathOperator{\GL}{{GL}}

\numberwithin{theorem}{section}
\theoremstyle{plain}

\begin{document}

\title[Bounding exotic $4$--manifolds]{On $3$--manifolds that are boundaries of exotic $4$--manifolds} 

\author{John B. Etnyre}

\author{Hyunki Min}

\author{Anubhav Mukherjee}

\address{School of Mathematics \\ Georgia Institute
of Technology \\  Atlanta  \\ Georgia}

\email{etnyre@math.gatech.edu}

\email{hmin38@gatech.edu}

\email{anubhavmaths@gatech.edu}


\begin{abstract}
We give several criteria on a closed, oriented $3$--manifold that will imply that it is the boundary of a (simply connected) $4$--manifold that admits infinitely many distinct smooth structures. We also show that any weakly fillable contact $3$--manifold, or contact $3$--manifold with non-vanishing Heegaard Floer invariant, is the boundary of a simply connected $4$--manifold that admits infinitely many distinct smooth structures each of which supports a symplectic structure with concave boundary, that is there are infinitely many exotic caps for any such contact manifold. 
\end{abstract}

\maketitle

\section{Introduction}

Quickly after the groundbreaking work of Freedman \cite{Freedman82} and Donaldson \cite{Donaldson83} in the early 1980s it was realized that closed $4$--manifolds could support more than one smooth structure. The first such example appeared in Donaldson's paper \cite{Donaldson85} after which it was shown by Okonek and Van de Ven \cite{OVV86} and Friedman and Morgan \cite{FM88}  that some topological $4$--manifolds admit infinitely many smooth structures. Since then there has been a great deal of work showing that many simply connected $4$--manifolds admit infinitely many smooth structures and it is possible that any $4$--manifold admitting a smooth structure admits infinitely many. 

A relative version of this phenomena has not been as well studied. Natural questions along these lines are the following. 
\begin{question}\label{q1}
Given a smooth $4$--manifold with boundary, does it admit infinitely many distinct smooth structures?
\end{question}
\noindent
We also have the following easier question. 
\begin{question}\label{q2}
Given a $3$--manifold, is it the boundary of a $4$--manifold that admits infinitely many smooth structures?
\end{question}

\subsection{Exotic fillings}
In this paper we give several partial answers to the second question.  The first is the following. 

\begin{theorem}\label{thm:main3}
    Let $Y$ be a closed, connected, oriented $3$--manifold. Suppose either $Y$ (or $-Y$) admits a contact structure $\xi$ such that its Heegaard Floer contact invariant $c^+(\xi)$ does not vanish. Then there exists a compact, simply connected $4$--manifold $X$ such that $\partial X=-Y$ (or $Y$) and $X$ admits infinitely many non-diffeomorphic smooth structures, each of which admits a symplectic structure that is a strong symplectic cap for $(Y,\xi)$ $(\text{or } (-Y,\xi))$.
\end{theorem}

We have the similar theorem.
\begin{theorem}\label{thm:main4}
Let $Y$ be a closed, connected, oriented $3$--manifold. Suppose either $Y$ (or $-Y$) has a weak symplectic filling $(W, \omega)$. Then there exists a compact, simply connected $4$--manifold $X$ such that $\partial X=-Y$ (or $Y$) and $X$ admits infinitely many non-diffeomorphic smooth structures relative to the boundary, each of which admits a symplectic structure such that $W\cup X$ is a closed symplectic manifold.  In addition, X also admits infinitely many smooth structures such that $W\cup X$ has no symplectic structure. 
\end{theorem}
\begin{remark}
Recall if the contact manifold $(Y,\xi)$ admits a strong symplectic filling, then Ozsv\'ath and Szab\'o \cite{OS04c} showed that $c^+(\xi)\not=0$ (see also \cite[Theorem~2.13]{Ghiggini06}), but if it just admits a weak filling it can happen that $c^+(\xi)=0$. So the two theorems above cover different situations. In \cite{OS04c}, Ozsv\'ath and Szab\'o also showed $c^+(\xi;[\omega])\neq0$ for a weakly fillable contact structure if one uses twisted coefficients, but unfortunately it is not clear how to extend Theorem~\ref{thm:main3} to this case. 
\end{remark}
\begin{remark}
Theorem~\ref{thm:main4} was proven by Yasui \cite{Yasui11} for strongly fillable $3$--manifolds and also for weak symplectic fillable manifolds if one does not require that $X$ be simply connected. Though in these cases he proved the stronger theorem that there were infinitely many smooth structures that were not related by any diffeomorphism. 
\end{remark}

\begin{remark}
In Theorem~\ref{thm:main4} the smooth structures are only shown to be exotic by a diffeomorphism that is the identity on the boundary. It would also be interesting to know the answer when $Y$ bounds an $X$ with infinitely many smooth structures that are not diffeomorphic by any diffeomorphism, we call these absolutely exotic structures. In Theorem~\ref{thm:main3} the smooth structures are absolutely exotic, while in Theorem~\ref{thm:main4} they are not. However, in many cases it is easy to see that they are also absolutely exotic. In particular, notice that if $Y$ has a finite number of diffeomorphisms up to isotopy, then an infinite subset of the smooth structures in Theorem~\ref{thm:main4} must be absolutely exotic. It is well-known that lens spaces \cite{Bonahon83}, hyperbolic manifolds \cite{GMT03}, and many other $3$--manifolds have finite mapping class groups. So for these manifolds one may remove ``rel boundary" from Theorem~\ref{thm:main4}.
\end{remark}

We can also say something for $3$--manifolds that embed in definite $4$--manifolds.
\begin{theorem}\label{thm:main5}
    Let $Y$ be an oriented rational homology $3$--sphere. Suppose $Y$ embeds as a separating hypersurface in a closed definite $4$--manifold. Then there exists a compact $4$--manifold $X$ with trivial first Betti number such that $\partial X=Y$ and $X$ admits infinitely many non-diffeomorphic smooth structures.
\end{theorem}

We will call a $4$--manifold $X$ with boundary $Y$ a filling of $Y$ and we will call a different smooth structure on $X$ an exotic filling, or an exotic smooth structure on the filling. So the above theorems can be stated as any closed oriented $3$--manifold satisfying the various hypothesis admits a filling with infinitely many exotic smooth structures (and in the case of the first two theorems the fillings can be taken to be simply connected). 
\begin{example}
We note that $P\# -P$, where $P$ is the Poincare homology sphere, does not admit a tight structure and hence Theorems~\ref{thm:main3} and~\ref{thm:main4} do not apply. But $P\# -P$ bounds $\overline{(P-B^3)} \times [0,1]$ which is a homology $4$--ball. Thus $P\# -P$ embeds in the double of the homology $4$--ball which is a homology $4$--sphere, and so it has a filling with infinitely many smooth structures by Theorem~\ref{thm:main5}. Similarly, any rational homology sphere that bounds a rational homology ball satisfies the hypothesis of Theorem~\ref{thm:main5}.  In particular, for any rational homology sphere $Y$, Theorem~\ref{thm:main5} applies to $Y\# -Y$. These examples were also constructed by Yasui \cite{Yasui11}.  
\end{example}

\begin{remark}
An obvious way to try to construct an infinite number of smooth structures on $X$ would be to start with one and form the connected sum with exotic smooth structures on a closed manifold. This may not produce distinct smooth structures. For example, if one chose $X$ to have many $S^2\times S^2$ summands, then connect summing $X$ with many families of exotic smooth manifolds will produce diffeomorphic manifolds by a result of Wall \cite{Wall64}. For a well-chosen $X$ it is likely that this construction will produce exotic fillings of $Y$, but proving they are different could be difficult as the Ozsv\'ath--Szab\'o and Seiberg--Witten invariants frequently vanish under connected sum. 
\end{remark}

\begin{remark}
We note that the above theorems guarantee that the following manifolds have a filling with infinitely many exotic smooth structures (though maybe not be absolutely exotic):
\begin{enumerate}
\item Any Seifert fibered space by \cite{Gompf98} and Theorem~\ref{thm:main3}.
\item Any $3$--manifold admitting a taut foliation by \cite{EliashbergThurston98, Gabai83, KazezRoberts19} and Theorem~\ref{thm:main4}.
\item Any irreducible $3$--manifold with positive first Betti number (this is a special case of the previous item).
\item Any rational homology $3$--sphere embedding in a closed definite manifold by Theorem~\ref{thm:main5}. 
\end{enumerate}
So the only irreducible $3$--manifolds not known to admit exotic fillings are hyperbolic homology spheres and toroidal homology spheres that do not embed in a definite $4$--manifold; and many of them are also know to have exotic fillings from the above theorems. In particular, we do not know any example of an irreducible $3$--manifold that does not satisfy the hypothesis of one of the theorems above. 

Moreover, if the $L$-space conjecture is true (at least the implication that an irreducible manifold that is not an $L$-space admits a taut foliation), then the only irreducible $3$--manifolds not known to admit an exotic filling will be $L$-space homology spheres. 
\end{remark}


Prior to this work there were several works addressing Question~\ref{q1} (and hence Question~\ref{q2}) in specific cases. We first note that since any diffeomorphism of $S^3$ extends over $B^4$, one can use all the past work on closed manifolds to show that $S^3$ has many fillings with infinitely many exotic smooth structures. Similarly, one can show that diffeomorphisms of circle bundles over surfaces can be extended over the disk bundles that they bound. Thus by finding embedded surfaces in the above mentioned closed manifolds with infinitely many smooth structures, one can remove neighborhoods of these surfaces to show that some circle bundles over surfaces have fillings with infinitely many exotic smooth structures. 

Moving beyond these obvious examples, the first result concerning Question~\ref{q1} is due to Gompf, \cite{Gompf91}. He showed that ``nuclei" $N_n$ of elliptic surfaces have infinitely many smooth structures. The manifold $N_n$ is a simply connected $4$--manifold with second homology of rank 2 and boundary the Brieskorn homology sphere $\Sigma(2,3,6n-1)$. Thus Question~\ref{q2} is answered for this family of $3$--manifolds. Later Yasui used these nuclei to give a general procedure to create exotic fillings of $3$--manifolds, \cite{Yasui11} see also \cite{Yasui14}. For example, he has shown that any symplectically fillable connected $3$--manifold 
bounds a compact connected oriented smooth $4$--manifold admitting infinitely many absolutely exotic smooth structures, as does the disjoint union of manifolds admitting Stein fillable contact structures. In fact proof of Theorem~\ref{thm:main4} can be seen as a corollary of the main theorem in \cite{Yasui11}.

In \cite{AEMS08}, Akhmedov, Mark, Smith and the first author gave explicit examples of circle bundles over surfaces with infinitely many exotic smooth fillings and also showed that all of these also have Stein structures. In particular, it was shown that there are contact $3$--manifolds that admit a filling by a $4$--manifold that has infinitely many smooth structures and each smooth structure has a Stein structure that fills the contact structure; that is, these contact manifolds admit infinitely many exotic Stein fillings.  This work was extended by Akhmedov and Ozbagci in \cite{AO14} to give families of Seifert fibered spaces with infinitely many exotic fillings (and exotic Stein fillings). More recently, Akbulut and Yasui \cite{AY14} gave an infinite family of $4$--manifolds with boundary each of which is simply connected, has second Betti number $b_2=2$, and admits infinitely many distinct smooth structures (each of which is also Stein). The fillings constructed in Theorems~\ref{thm:main3}, \ref{thm:main4}, and~\ref{thm:main5} are large in the sense that in general they have large $b_2$. In the closed case it is quite interesting to try to find the simply connected $4$--manifolds with the smallest $b_2$ that have infinitely many smooth structures. This leads us to naturally ask the following question.
\begin{question}
Given a closed oriented $3$--manifold $Y$, what is the minimal second Betti number of a filling of $Y$ that admits infinitely many exotic smooth structures?
\end{question}

We also note that in \cite{AY13}, Akbulut and Yasui show that many $3$--manifolds (in particular ones realizing all possible homologies for a $3$--manifold) admit fillings with any arbitrarily large, but finite number of smooth structures (that also admit Stein structures).

\subsection{Concordance surgery and Ozsv\'ath--Szab\'o invariants}
The proof of the theorems in the previous section will rely on a construction called concordance surgery and the effect of this surgery on the Heegaard Floer mixed map. 

\dfn{Concordance surgery} is a generalization of Fintushel-Stern's knot surgery \cite{FS98}; see \cite{Akbulut02, JZ18, Tange05}. Let $X$ be a $4$--manifold and $T$ a torus embedded in $X$ with trivial normal bundle. Thus, $T$ has a neighborhood $N_T = T\times D^2$. Let $K$ be a knot in an integer homology $3$--sphere $M$ and $\cC = (I\times M, A)$ a self-concordance from $K$ to itself. After gluing the ends of $A$ together, we obtain an embedded torus $T_\cC$ in $S^1\times M$. Let $W_\cC$ be the complement of a neighborhood of $T_\cC$. Choose an orientation reversing diffeomorphism $\phi:\partial (X\setminus N_T)\rightarrow \partial W_\cC$ sending $\partial D^2$ in $N_T$ to the longitude for $K$ in $Y\setminus N_K$. Now we may glue $X\setminus N_T$ and $W_\cC$ together and obtain
\[
    X_\cC := (X\setminus N_T)\cup_\phi W_\cC.
\]
Knot surgery is simply the case of concordance surgery when $M=S^3$ and $A=K\times [0,1]$. 

Fintushel and Stern \cite{FS98} showed how knot surgery affects the Seiberg--Witten invariants of a $4$--manifold (under suitable hypothesis), but it was a long-standing question how concordance surgery affects the Seiberg--Witten invariants. In \cite{JZ18}, Juh\'{a}sz and Zemke computed the effects of concordance surgery on the Ozsv\'{a}th--Szab\'{o} invariants of a closed $4$--manifold, which are conjecturally same as the Seiberg--Witten invariants.

In this paper, we will show a similar result for $4$--manifolds with connected boundary. (We can prove our main results using just knot surgery, but as formulas for the effect of knot and concordance surgery on the Heegaard-Floer invariants of manifolds with boundary do not exist and their proofs are almost the same, we prove the results for the more general surgery.) To state the result, we need to define the Ozsv\'{a}th--Szab\'{o} polynomial for $4$--manifolds with boundary; see also \cite{JM08, JZ18} for closed $4$--manifolds. Suppose $X$ is a smooth, compact, oriented $4$--manifold with connected boundary $Y$ and $b^+_2(X)>1$. Generalizing the Ozsv\'{a}th--Szab\'{o} invariant from closed manifolds to manifolds with non-empty boundary we say the Ozsv\'{a}th--Szab\'{o} invariant is a map
\[
    \Phi_X:\Spin^c(X)\rightarrow HF^+(Y).
\]
This is an invariant of $X$ up to automorphisms of $\Spin^c(X)$ and $HF^+(Y)$.
Write $\Phi_{X,\frs}$ for $\Phi_X(\frs)$, which is the image of the bottom-graded generator of $HF^-(S^3)$ under the mixed map associated to the cobordism $\overline{X-B^4}$ thought of as a cobordism from $S^3$ to $Y$. For more details, see Section~\ref{mixed}. Suppose $\b=(b_1, \ldots , b_n)$ is a basis of $H^2(X,\partial{X};\R)$ and $\frs_0$ is a fixed $\Spin^c$ structure on $X$. We define the Ozsv\'{a}th--Szab\'{o} polynomial for $X$ as follows.

\[
    \Phi_{X;\b} := \sum_{\frs \in \Spin^c(X)} z_1^{\langle i_*(\frs-\frs_0)\cup b_1,[X,\partial X] \rangle}\cdots z_n^{\langle i_*(\frs-\frs_0)\cup b_n,[X,\partial X] \rangle} \cdot \Phi_{X,\mathfrak{s}},
\]
which is an element in $\F_2[\Z^n]\otimes_{\F_2} HF^+(Y)$, where $i_*:H^2(X;\Z)\rightarrow H^2(X;\R)$ is a map induced by $i:\Z\rightarrow\R$. If $H^2(X)$ is torsion-free, $\Phi_{X;\b}$ completely encodes $\Phi_{X}$ as in the closed case.

Let $K$ be a knot in a homology sphere $M$ and $\cC$ be a self-concordance of $K$. In \cite{JZ18}, Juh\'{a}sz and Zemke defined the \dfn{graded Lefschetz number of $\cC$} as follows.
\[
    \Lef_z(\cC) := \sum_{i\in\Z}\Lef\left(\widehat{F}_\cC|_{\widehat{HFK}(M,K,i)}:\widehat{HFK}(M,K,i)\rightarrow\widehat{HFK}(M,K,i)\right)\cdot z^i,
\]
where $\widehat{F}_\cC:\widehat{HFK}(M,K)\rightarrow\widehat{HFK}(M,K)$ is a concordance map defined by Juh\'{a}sz in \cite{Juhasz16}, which preserves the Alexander and Maslov gradings. Notice that if $\cC$ is a product concordance, then $\widehat{F_\cC}$ is the identity, so $\Lef_z(\cC)$ is the Alexander polynomial $\Delta_K(z)$. Our main result about knot concordance is the following result that generalizes Juh\'{a}sz and Zemke's result above from closed manifolds to manifolds with boundary. 

\begin{theorem}\label{thm:main1}
    Let $X$ be a smooth, compact, oriented $4$--manifold with connected boundary $Y=\partial X$ such that $b_2^+(X) \geq 2$. Suppose $T$ is an embedded torus in $X$ with trivial normal bundle such that $[T]\neq0\in H_2(X;\R)$ and $\b=(b_1,\ldots,b_n)$ is a basis of $H^2(X,\partial X;\R)$ such that $\langle[T],b_1\rangle=1$ and $\langle[T],b_i\rangle=0$ for $i>1$. Then
    \[
        \Phi_{X_\cC;\b} = \Lef_{z_1}(\cC) \cdot \Phi_{X;\b}.
    \]
\end{theorem}
A corollary of this theorem is the following useful result that will be the key to proving Theorems~\ref{thm:main3} and~\ref{thm:main5}
\begin{corollary}\label{cor:main2}
    Suppose  $\cC$ and $\cC'$ are concordances such that $\Lef_{z}(\cC)\not= \Lef_{z}(\cC')$. If $\Phi_{X;\b} \neq 0$, then $X_\cC$ is not diffeomorphic to $X_{\cC'}$.    
\end{corollary}
The proof of this corollary follows an argument of Sunukjian \cite{Sunukjian15}, and the authors are grateful to Gompf for pointing out that we might be able to use such an argument. We note that this strengthens Corollary~1.2 in \cite{JZ18}.

\begin{remark}
To prove our main result about exotic fillings of $3$--manifolds we will only need to use knot surgery and not concordance surgery, but thought it was useful to develop the effects of these surgeries on the Ozsv\'{a}th--Szab\'{o} polynomial in as great a generality as possible. 
\end{remark}

Given the above results our main theorem, Theorem~\ref{thm:main3}, will follow from the following result. 

\begin{theorem}\label{cap}
Given any closed connected contact $3$--manifold $(Y,\xi)$, there is a (strong) symplectic cap $(X,\omega)$ for $(Y,\xi)$ that is simply connected and contains a Gompf nucleus $N_2$ whose regular fiber is symplectic and has simply connected complement. Moreover, $\Phi_{X,\frs_0}=c^+(\xi)\in HF^+(-Y,\frs_0|_Y)$ for the canonical $\Spin^c$ structure $\frs_0$ on $(X,\omega)$.  
\end{theorem}

The first part of the theorem is almost proven in \cite{AO02, EH02a, Gay02} and, in different language, proven in \cite{Yasui11}, and the last part was proven by Plamenevskaya \cite{Plamenevskaya04} and Ghiggini \cite{Ghiggini06}, but as the argument is simple and we need caps with all the listed properties we sketch the proof in Section~\ref{caps}. 

This theorem brings up several interesting questions about how much of the Heegaard Floer homology of a $3$--manifold can be seen from the $4$--manifolds it bounds. More specifically we ask the following questions.
\begin{question}
Given a $3$--manifold $Y$ and a homogeneous element $\eta\in HF^+(Y)$ is there a $4$--manifold $X$ with $\partial X=Y$ and $\Phi_{X,\frs}=\eta$ for some $\Spin^c$ structure $\frs$?
\end{question}
Or more simply we can ask:
\begin{question}\label{q8}
Given a $3$--manifold $Y$ is there any $4$--manifold $X$ such that $\partial X=Y$ and $\Phi_{X,\frs}$ is a non-zero element of $HF^+(Y)$ for some $\Spin^c$ structure $\frs$?
\end{question}
If the answer to either question is yes, then one might ask if one can also control the topology of $X$. For example can it be chosen to be simply connected? Can it be chosen to contain a cusp neighborhood?

One way to approach Question~\ref{q8} is related to a pervious attempt by the authors to show that all closed oriented $3$--manifolds bound a simply connected manifold with infinitely many smooth structures by trying to embed such a manifold in a closed symplectic manifold. While our construction did not work, as pointed out by Gompf, it does lead to an interesting question which we state as a conjecture.
\begin{conjecture}
Any closed, oriented $3$--manifold admits a smooth embedding into a symplectic $4$--manifold. 
\end{conjecture}
\begin{remark}
In the paper \cite{Mukerjee20pre}, the third author showed that any such $3$--manifold admits a topological embedding in a symplectic manifold that can be made to be smooth after a single connected sum with $S^2\times S^2$. 
\end{remark}

\subsection*{Organization.} In Section~\ref{contactprelim} we review results concerning contact structures and Weinstein cobordisms. Perturbed Heegaard Floer theory is reviewed in Section~\ref{perturbedHF} with a focus on the work of Juh\'{a}sz and Zemke \cite{JZ18}. In Sections~ \ref{KCS} and~\ref{concsurgery}, we review Juh\'{a}sz and Zemke's work \cite{JZ18} on how concordance surgery affects the Ozsv\'ath--Szab\'o invariants of a closed $4$--manifold and prove our Theorem~\ref{thm:main1} on its affect on the invariants of manifolds with boundary. Theorems~\ref{thm:main3} and~\ref{thm:main4} are proven in Section~\ref{caps}, while Theorem~\ref{thm:main5} is proven in Section~\ref{df}.

\subsection*{Acknowledgements}
The authors thank Robert Gompf and Kouichi Yasui for pointing out an error in the first version of the paper. We also thank Tom Mark and Ian Zemke for helpful discussions about the Ozsv\'{a}th--Szab\'{o} invariants, Chris Gerig and Olga Plamenevskaya for helpful discussions about symplectic caps, Jen Hom and Jaewoo Jung for helpful conversations about Heegaard Floer homology, and Michael Klug for asking the question that prompted this research. We also thank Marco Golla and an anonymous referee for helpful comments on an earlier draft of the paper. The authors were partially supported by NSF grant DMS-1608684 and and DMS-1906414.

\section{Contact structures, Weinstein cobordisms, and open books}\label{contactprelim}
We assume the reader is familiar with basic results concerning contact and symplectic geometry, convexity and concavity of symplectic manifolds with boundary, and open book decompositions as can be found in \cite{Etnyre06}, but we briefly recall some of this to establish notation and for the convenience of the reader. 

We begin by recalling that a Legendrian knot $L$ in a contact manifold $(Y,\xi)$ has a standard neighborhood $N$ and a framing $fr_\xi$ given by the contact planes. If $L$ is null-homologous then $fr_\xi$ relative to the Seifert framing is the Thurston-Bennequin invariant of $L$.  If one does $fr_\xi-1$ surgery on $L$ by removing $N$ and gluing back a solid torus so as to affect the desired surgery, then there is a unique way to extend $\xi|_{Y-N}$ over the surgery torus so that it is tight on the surgery torus. The resulting contact manifold is said to be obtained from $(Y,\xi)$ by Legendrian surgery on $L$. 

Recall a symplectic cobordism from the contact manifold $(Y_-,\xi_-)$ to $(Y_+,\xi_+)$ is a symplectic manifold $(W,\omega)$ with boundary $-Y_-\cup Y_+$ where $Y_-$ is the the part of the boundary that is concave and $Y_+$ is convex. Here, unless specifically stated otherwise, we mean convex and concave in the strong sense. The first result we will need concerns when symplectic cobordisms can be glued together. 
\begin{lemma}\label{lem:glue}
If $(X_i,\omega_i)$ is a symplectic cobordism from $(Y_i^-,\xi_i^-)$ to $(Y_i^+,\xi_i^+)$, for $i=1,2$, and $(Y_1^+,\xi_1^+)$ is contactomorphic to $(Y_2^-, \xi_2^-)$, then we may use the contactomorphism to glue $X_1$ and $X_2$ together to get a symplectic cobordism from $(Y_1^-,\xi_1^-)$ to $(Y_2^+,\xi_2^+)$.
\end{lemma}
\noindent 
The proof is a simple exercise, {\em cf.\ }\cite{Etnyre98}. 

Another way to build cobordisms is by Weinstein handle attachment, \cite{GS99, Weinstein91}. One may attach a 0, 1, or 2--handle to the convex end of a symplectic cobordism to get a new symplectic cobordism with the new convex end described as follows. For a 0--handle attachment, one merely forms the disjoint union with a standard $4$--ball and so the new convex boundary will be the old boundary disjoint union with the standard contact structure on $S^3$. For a 1--handle attachment, the convex boundary undergoes a, possibly internal, connected sum. A 2--handle is attached along a Legendrian knot $L$ with framing one less that the contact framing, and the convex boundary undergoes a Legendrian surgery. 

Given a surface $\Sigma$ with boundary and a diffeomorphism $\phi:\Sigma\to \Sigma$ that is the identity near $\partial \Sigma$ we can form a closed $3$--manifold $M_{(\Sigma,\phi)}$ by gluing solid tori to the boundary of the mapping torus
\[
T_\phi = \Sigma\times [0,1]/\sim,
\]
where $(1,x)\sim(0,\phi(x))$, so that the meridians of the solid tori map to $\{p\}\times [0,1]/\sim$ for some $p\in \partial \Sigma$. If $M$ is diffeomorphic to $M_{(\Sigma,\phi)}$ then we say that $(\Sigma, \phi)$ is an open book decomposition for $M$. Following work of Thurston and Winkelnkemper \cite{TW75}, Giroux \cite{Giroux02} showed that there is a unique contact structure associated to an open book decomposition. We say that the contact structure is supported by the open book. Moreover, he also showed that every contact structure is supported by some open book decomposition and if an open book is positively stabilized, then the supported contact structure is the same. A positive stabilization of $(\Sigma, \phi)$ is $(\Sigma',\phi')$ where $\Sigma'$ is obtained form $\Sigma$ by attaching a 2--dimensional 1--handle and $\phi'=\phi\circ \tau_\gamma$, where $\gamma$ is a curve on $\Sigma'$ that intersects the co-core of the 1-handle transversely and exactly once, and $\tau_\gamma$ is a right handed Dehn twist about $\gamma$.

The following theorem is proven in \cite{Etnyre06}.
\begin{theorem}\label{thm:2h}
Let $(Y,\xi)$ be a strongly (respectively weakly) convex boundary component of a symplectic manifold $(X, \omega)$ and $(\Sigma, \phi)$ an open book decomposition supporting $\xi$. If $\gamma$ is a non-separating curve on $\Sigma$, then 
\begin{enumerate}
\item a page of the open book may be isotoped, so that the open book still supports $\xi$, and $\gamma$ becomes a Legendrian curve,
\item a Weinstein 2--handle may be attached to $\gamma$ resulting in a new symplectic manifold $(X',\omega')$ whose new boundary component $(Y',\xi')$ is strongly (respectively weakly) convex and obtained from $(Y,\xi)$ by Legendrian surgery on $\gamma$, and
\item $(Y',\xi')$ is supported by the open book $(\Sigma, \phi\circ \tau_\gamma)$. 
\end{enumerate}
\end{theorem}

\section{Perturbed Heegaard Floer homology}
In order to prove Theorem~\ref{thm:main1} we will need to use perturbed Heegaard Floer homology that was originally defined by Ozsv\'ath and Szab\'o in  \cite{OS04a}, {\em cf}  \cite{JM08}. The dependence of this homology on the data used to define it is rather subtle and was worked out in Juh\'{a}sz and Zemke's paper {\cite[Theorem~3.1]{JZ18}} using the formalism of projective transitive systems that was introduced by Baldwin and Sivek in \cite{BS15} when studying the naturality of sutured monopole and instanton homology.

\subsection{Novikov rings and projective transitive systems}\label{novikov}
Recall that 
    the \dfn{Novikov ring} over $\F_2$ is a set of formal series 
    \[
        \Lambda = \left\{\sum_{x\in\R}n_x z^x : n_x\in\F_2\right\} 
    \] 
    where the set 
    \[
        \{x \in (-\infty, c] : n_x\neq0\}
    \] 
    is finite for every $c\in\R$. One may easily check that this is a field under the obvious operations. 

Let $Y$ be a closed $3$--manifold equipped with a closed $2$-form $\omega\in\Omega^2(Y)$. Then there exists an action of a group ring $\F_2[H^1(Y)] \cong \F_2[H_2(Y)]$ on $\Lambda$, induced by $\omega$, which is defined as follows. For $a \in H_2(Y)$,

\[
    e^a \cdot z^x = z^{x+\int_a\omega}.
\] 

The naturality of perturbed Heegaard Floer homology is conveniently described by projective transitive systems, which were first introduced by Baldwin and Sivek in \cite{BS15}.

\begin{definition}
    Let $\cC$ be a category and $I$ a set. A \dfn{transitive system in $\cC$ indexed by I} consists of 
    \begin{itemize}
        \item a collection of objects $(X_i)_{i \in I}$ in $\cC$ and
        \item distinguished morphisms $\Phi^i_j: X_i \rightarrow X_j$ for $(i,j) \in I \times I$ such that
        \begin{enumerate}
            \item $\Phi^j_k \circ \Phi^i_j = \Phi^i_k$ and
            \item $\Phi^i_i = id_{X_i}$.
        \end{enumerate}
    \end{itemize}
\end{definition}

Let $\cC$ be the projectivized category of $\Lambda[U]$--modules, where $\Lambda$ is the Novikov ring above and $U$ is a formal variable. The objects of $\cC$ are $\Lambda$--modules and the morphism set $\Hom_{\cC}(X_1,X_2)$ is the projectivization of $\Hom_{\Lambda}(X_1,X_2)$ under the action by left multiplication of elements of $\Lambda$ that are of the form $z^x$.

In \cite{BS15, BS16}, Baldwin and Sivek called a transitive system over the projectivized category a \dfn{projective transitive system}. For morphisms $f,g\in \Hom_{\Lambda}(X_1,X_2)$, we write $f \doteq g$ if $f = z^x \cdot g$ for some $x \in \R$. Also, if $X$ is a $\Lambda$--module and $a,b \in X$, we write $a \doteq b$ if $a = z^x \cdot b$. Note that the composition of morphisms in a projective transitive system is well-defined, while the addition of morphisms is not well-defined.

There is also a notion of a morphism between transitive systems. 

\begin{definition} Let $(X_i)_{i\in I}$ and $(Y_j)_{j\in J}$ be transitive systems in the category $\cC$. A \dfn{morphism of transitive systems} is a collection of morphisms
    \[
        F^i_j:X_i\rightarrow Y_j
    \]
    in $\cC$ such that
    \[
        \Phi^j_{j'}\circ F^i_j\circ\Phi^{i'}_i = F^{i'}_{j'}
    \]
    for all $i, i'\in I$ and $j, j'\in J$.
\end{definition}

\subsection{Perturbed Heegaard Floer homology}\label{perturbedHF}


Let $Y$ be a closed $3$--manifold and $\frs$ a $\Spin^c$ structure on $Y$. Recall that Heegaard Floer homology is a package of $\F_2[U]$-modules $HF^\circ(Y,\frs)$ for $\circ\in\{\infty,+,-,\wedge\}$, which fit into a long exact sequence
\[
    \cdots\xrightarrow{\tau}HF^-(Y,\frs)\rightarrow HF^\infty(Y,\frs)\rightarrow HF^+(Y,\frs)\xrightarrow{\tau} HF^-(Y,\frs)\rightarrow\cdots
\]

In \cite{OS04a}, Ozsv\'{a}th--Szab\'{o} introduced Heegaard Floer homology perturbed by a second real cohomology class, which is more thoroughly discussed in \cite{JM08}. Let $\omega\in\Omega^2(Y)$ be a closed 2-from on $Y$ and $\cH = (\Sigma,\alphas,\betas,w)$ an $\frs$-admissible pointed Heegaard diagram of $Y$. Denote the two handlebodies determined by $\cH$ by $H_{\alphas}$ and $H_{\betas}$ respectively and let $D_{\alphas}$ and $D_{\betas}$ be sets of compressing disks of $H_{\alphas}$ and $H_{\betas}$, respectively, such that $D_{\alphas}$ intersects $\Sigma$ along $\alphas$ and $D_{\betas}$ intersects $\Sigma$ along $\betas$. To define perturbed Heegaard Floer homology, we need to keep track of homotopy data associated to $\phi\in\pi_2(\xs,\ys)$. Note that $\phi$ determines a 2--chain $\cD(\phi)$ on $\Sigma$ with boundary a union the loops in $\alphas\cup \betas$. One may cone these loops in the compressing disks $D_{\alphas}$ and $D_{\betas}$ to obtain a 2--chain $\widetilde{D}(\phi)$. Now we define
\[
    A_\omega(\phi):=\int_{\widetilde{\cD}(\phi)}\omega.
\]

Consider a chain complex $CF^\infty(\cH,\frs;\omega)$ which is a free $\Lambda$-module generated by $U^i\xs$ for $\xs\in T_\alpha\cap T_\beta$, such that the $\Spin^c$ structure associated to $\xs$ is $\frs$, i.e. $\frs(\xs)=\frs$. The differential is defined as follows.

\[
    \partial^\infty(U^i\xs) = \sum_{y\in T_\alpha\cap T_\beta}\sum_{\substack{\phi\in\pi_2(x,y) \\ \mu(\phi)=1}}\#(\cM(\phi)/\R)\cdot z^{A_\omega(\phi)}\cdot U^{i-n_w(\phi)}\ys \quad \text{(mod 2)},
\]
where $n_w(\phi)$ is the algebraic intersection number between $\phi\in\pi_2(\xs,\ys)$ and $\{w\}\times Sym^{g-1}(\Sigma)$ and $\cM(\phi)$ is the space of $J$-holomorphic disks in the homotopy class $\phi$ for $J \in \mathcal{J}$, where $\mathcal{J}$ is a generic family of almost complex structures on $Sym^g(\Sigma)$. For $\circ\in\{+,-,\wedge\}$, $\partial^\circ$ is induced from $\partial^\infty$ in the usual way and $(\partial^\circ)^2=0$ as in the original Heegaard Floer homology. Now we define 
\[
    HF^\circ(Y,\frs;{\omega}):=H_*(CF^\circ(\cH,\frs;\omega),\partial^\circ).
\]
The homology $HF^\circ(Y,\frs;{\omega})$ depends on the choice of $\cH$ and $J$, but it forms a projective transitive system.

\begin{theorem}[Juh\'{a}sz--Zemke \protect{\cite[Theorem~3.1]{JZ18}}]
    Let $Y$ be a $3$--manifold equipped with a closed $2$-form $\omega \in \Omega^2(Y)$ and $\frs \in \Spin^c(Y)$ a $\Spin^c$ structure on $Y$. For $\circ \in \{\infty, +, -, \wedge\}$, $HF^\circ(Y,\frs;{\omega})$ forms a projective transitive system of $\Lambda[U]$-modules, indexed by the set of pairs $(\cH,J)$, where $\cH = (\Sigma, \alphas, \betas, w)$ is an $\frs$-admissible pointed Heegaard diagram of $Y$, and $J$ is a generic almost complex structure on $Sym^g(\Sigma)$.
\end{theorem}

We also have functoriality as in unperturbed Heegaard Floer homology. 

\begin{theorem}[Ozsv\'ath--Szab\'o~\protect{\cite[Section~3.1]{OS04c}}]
    Let $W$ be a cobordism from $Y_1$ to $Y_2$. Suppose $\omega$ is a closed $2$-form on $W$ and $\frs\in\Spin^c(W)$ is a $\Spin^c$ structure on $W$. Then the cobordism map
    \[
        F^\circ_{W,\frs;\omega}\colon HF^\circ(Y_1,\frs|_{Y_1};{\omega|_{Y_1}})\to HF^\circ(Y_2,\frs|_{Y_2};{\omega|_{Y_2}})
    \]
    is well-defined up to overall multiplication by $z^x$ for $x\in\R$.
\end{theorem}

Consider $\omega\in\Omega^2(W,\partial{W})$ which is a closed $2$-form on $W$ compactly supported in the interior of $W$. We will say such an $\omega$ is a $2$-form on $(W,\partial{W})$. Let $W$ be a cobordism from $Y_0$ to $Y_1$ equipped with a closed $2$-form $\omega$ and $\frS\subseteq\Spin^c(W)$ a subset of $\Spin^c$ structures on $W$. From now on, we always assume that $(W,\omega,\frS)$ has one of the following properties.
\begin{itemize}
    \item each $\frs\in\frS$ has the same restriction to $\partial{W}$, or
    \item $\omega$ is a closed $2$-form on $(W,\partial{W})$.
\end{itemize}
If $\circ\in\{\infty, -\}$, we further assume that there exists only finitely many $\frs\in\frS$ such that $F^\circ_{W,\frs;\omega}$ is non-vanishing. Then, there exists a cobordism map
\[
    F^\circ_{W,\frS;\omega}: HF^\circ(Y_1;{\omega|_{Y_1}})\rightarrow HF^\circ(Y_2;{\omega|_{Y_2}}),
\]
which is also well-defined up to overall multiplication by $z^x$ for $x\in\R$. Although addition is not well-defined in projective systems, we may find representatives of $F^\circ_{W,\frS;\omega}$ for $\frs\in\frS$ so that
\[
F^\circ_{W,\frS;\omega} \doteq \sum_{\frs\in \frS} F^\circ_{W,\frs;\omega}.
\]

There is also a $\Spin^c$ composition law for cobordism maps.
\begin{lemma}[Ozsv\'{a}th--Szab\'{o} \protect{\cite[Theorem~3.4]{OS06}}]\label{prop:composition}
    Let $W$ be a cobordism which is decomposed into $W = W_1 \cup W_2$. Suppose that $\omega$ is a closed $2$-form on $(W,\partial{W})$, and $\frS_1 \subset \Spin^c(W_1)$ and $\frS_2 \subset \Spin^c(W_2)$ are subsets of $\Spin^c$ structures satisfying the properties above. Let
    \[
    \frS(W,\frS_1,\frS_2) = \{\, \frs \in \Spin^c(W) :
        \frs|_{W_1} \in \frS_1 \text{ and } \frs|_{W_2} \in \frS_2 \,\}.
    \]
    Then
    \[
    F^\circ_{W,\frS(W,\frS_1,\frS_2);\omega}
    \doteq F^\circ_{W_2, \frS_2; \omega|_{W_2}} \circ
    F^\circ_{W_1, \frS_1; \omega|_{W_1}}.
    \]    
\end{lemma}

The cobordism map is unchanged when we replace $\omega$ by $\omega+d\eta$ for any $\eta$.
\begin{lemma}[Juh\'{a}sz--Zemke \protect{\cite[Lemma~3.3]{JZ18}}]\label{lem:changing2form}
    Let $W$ be a cobordism from $Y_1$ to $Y_2$, $\frS \subset \Spin^c(W)$ a subset of $\Spin^c$ structures on $W$, $\omega$ a closed $2$-form on $(W,\partial{W})$, and $\eta$ a 1-form on $(W,\partial{W})$. Then
    \[
    F^\circ_{W, \frS; \omega} \dot{=} F_{W, \frS; \omega + d\eta}.
    \]
    If $\omega$ does not vanish on $Y_1$ and $Y_2$, then the above equation holds when restricted to fixed $\Spin^c$ structures on $Y_1$ and $Y_2$.
\end{lemma}

If $\omegas=\{\omega_1,\ldots,\omega_n\}$ is an $n$-tuple of closed $2$-forms on a $3$--manifold $Y$, we can define a $\Lambda_n[U]$-module $HF^\circ(Y,\frs;{\omegas})$ as above, where $\Lambda_n$ is the $n$-variable Novikov ring over $\F_2$. All the theorems and lemmas in this section hold for this version.

Let $\a = (a_1,\ldots,a_n)$ be an $n$-tuple of integers. We will use the notation
\[
    \z^{\a} := z_1^{a_1} \cdots z_n^{a_n}.
\]

\begin{lemma}[Juh\'{a}sz--Zemke \protect{\cite[Lemma~3.4]{JZ18}}]\label{lem:triviallytwisted}
    Let $W$ be a cobordism from $Y_1$ to $Y_2$ and $\omegas=\{\omega_1,\ldots,\omega_n\}$ be an $n$-tuple of closed $2$-forms on $(W,\partial{W})$. Suppose $\frS \subseteq \Spin^c(W)$ is a subset of $\Spin^c$ structures on $W$. If $\circ \in \{-,\infty\}$, we further assume that there are only finitely many $\frs \in \frS$ where $F^\circ_{W,\frs} \neq 0$. Fix an arbitrary $\Spin^c$ structure $\frs_0 \in \Spin^c(W)$. Then,
    \[
        F^\circ_{W,\frS;\omegas}\doteq\sum_{s\in\frS} \z^{\langle i_*(\frs-\frs_0)\cup[\omegas],[W,\partial W]\rangle} \cdot F^\circ_{W,\frs},
    \]
where $i_*:H^2(W;\Z)\to H^2(W;\R)$ is induced by the inclusion $i:\Z\to\R$. 
\end{lemma}


In \cite{OS06}, Ozsv\'{a}th--Szab\'{o} defined a pairing for Heegaard Floer homology

\[
    \langle\cdot,\cdot\rangle: HF^+(Y,\frs)\otimes HF^-(-Y,\frs) \rightarrow \Z,
\]
which is non-degenerate for torsion $\Spin^c$ structures and satisfies a certain duality property.

In \cite[Sections~4 and~8.1]{JM08}, Jabuka and Mark extended this pairing to Heegaard Floer homology with twisted coefficients and to perturbed Heegaard Floer homology. 
Let $M$ be a $\Lambda$-module. We denote the additive group $M$ equipped with conjugate module structure by $\overline{M}$, where multiplication is given by
\[
    z \otimes a \mapsto z^{-1} \cdot a,    
\]
for $z\not= 0 \in \Lambda$ and $a \in \overline{M}$.

\begin{theorem}[Jabuka--Mark \cite{JM08}]\label{thm:pairing}
    There is a non-degenerate pairing
    \[
        \langle\cdot,\cdot\rangle:HF^+(Y,\frs;\omega)\otimes_{\Lambda}\overline{HF^-(-Y,\frs;\omega)}\rightarrow\Lambda.
    \]     
    The pairing satisfies $\langle g a, b\rangle = g\langle a, b\rangle =\langle  a, {g} b\rangle$ for $g\in \Lambda$. 
\end{theorem}

The pairing also satisfies the following duality property.

\begin{theorem}[Jabuka--Mark \cite{JM08}]\label{thm:duality}
    Let $W$ be a cobordism from $Y_1$ to $Y_2$, $\omega$ a closed $2$-form on $W$ and $\frs$ a $\Spin^c$ structure on $W$. For $a\in HF^+(Y_1,\frs|_{Y_1};{\omega|_{W_1}})$ and $b\in HF^-(-Y_2,\frs|_{Y_2};{\omega|_{W_2}})$, we have
    \[
        \langle F^+_{W,\frs;\omega}(a),b\rangle=\langle a,F^-_{W',\frs;\omega}(b)\rangle,
    \]
    where we consider $W'$ is $W$ viewed as a cobordism from $-Y_2$ to $-Y_1$. 
\end{theorem}

\subsection{Ozsv\'{a}th--Szab\'{o} mixed invariants}\label{mixed}

Let $W$ be a cobordism from $Y_1$ to $Y_2$ with $b_2^+(W)>1$. In \cite{OS06}, Ozsv\'{a}th and Szab\'{o} defined the \dfn{mixed map}

\[
  F^{mix}_{W,\frs}: HF^-(Y_1,\frs|_{Y_1}) \rightarrow HF^+(Y_2,\frs|_{Y_2}).
\]

To define the mixed map, we choose a codimension one submanifold $N$ which separates $W$ into two cobordisms $W_1$ and $W_2$ and satisfies
\begin{itemize}
    \item $b_2^+(W_i)>0$ for $i=1,2$ and
    \item $\delta: H^1(N) \rightarrow H^2(W,\partial W)$ vanishes
\end{itemize}
Such an $N$ is called an \dfn{admissible cut}. In \cite{OS06}, Ozsv\'{a}th and Szab\'{o} proved that $F^\infty_{W_1,\frs|_{W_1}}=0$ and $F^\infty_{W_2,\frs|_{W_2}}=0$ (this is also true for perturbed cobordism maps). Thus, we may restrict the codomain of $F^-_{W_1,\frs|_{W_1}}$ to 
\[
    HF^-_{red}(N,\frs|_N) := \ker\left(HF^-(N,\frs|_N) \rightarrow HF^\infty(N,\frs|_N)\right),
    \]
    and thus get a map
    \[
    F^-_{W_1,\frs|{W_1}}:HF^-(Y_1,\frs|_{W_1}) \rightarrow HF^-_{red}(N,s|_N).
\]
Moreover, 
if 
\[
    HF^+_{red}(N,\frs|_N) := \mathrm{coker}\left(HF^\infty(N,\frs|_N) \rightarrow HF^+(N,\frs|_N)\right),
    \]
    then we may factor $F^+_{W_2,\frs|_{W_2}}$ through $ HF^+_{red}(N,\frs|_N)$ to get
    \[
    F^+_{W_2,\frs|_{W_2}}:HF^+_{red}(N,\frs|_N) \rightarrow HF^+(Y_2,\frs|_{W_2}).
\]

The boundary map $\tau$ in the long exact sequence of Heegaard Floer homology (see Section~\ref{perturbedHF}) induces an isomorphism between  $H^+_{red}(N,\frs|_N)$ and $H^-_{red}(N,\frs|_N)$.

Now the mixed map is defined as follows:
\[
    F^{mix}_{W,\frs} := F^+_{W_2,\frs|_{W_2}} \circ \tau^{-1} \circ F^-_{W_1,\frs|_{W_1}}.
\]
In \cite{OS06}, Ozsv\'{a}th and Szab\'{o} proved that the mixed map is independent of the admissible cut $N$.

Let $X$ be a compact $4$--manifold with connected boundary $Y$ and $b_2^+(X)>1$. We consider $X$ as a cobordism from $S^3$ to $Y$. We define the \dfn{Ozsv\'{a}th--Szab\'{o} invariant of $X$} to be the map
\[
    \Phi_{X} : \Spin^c(X) \rightarrow HF^+(Y) 
\]
that sends $\frs$ to
\[
    \Phi_{X,\frs} := F^{mix}_{X,\frs}(\theta^-) 
\]
where $\theta^-$ is the top-graded generator of $HF^-(S^3)$.

\begin{remark}
    In \cite{OS06} Ozsv\'{a}th and Szab\'{o} defined $\Phi_{X,\frs}$ as a numerical invariant for a closed $4$--manifold $X$ by pairing $F^{mix}_{X,\frs}(\theta^-)$ with $\theta^+$, which is the bottom-graded generator for $HF^+(S^3)$. However, since there is no canonical element for general $HF^+(Y)$, we define $\Phi_{X,\frs}$ as an element in $HF^+(Y)$. 
\end{remark}

When $X$ is a closed $4$--manifold, Jabuka and Mark in \cite{JM08}, and Juh\'{a}sz and Zemke in \cite{JZ18} computed the Ozsv\'{a}th--Szab\'{o} invariant of $X$ using perturbed cobordism maps. The following lemmas and properties are proved by Juh\'{a}sz and Zemke in \cite{JZ18} when $X$ is a closed $4$--manifold. The proofs are identical for the relative case.

\begin{lemma}[Juh\'{a}sz--Zemke \protect{\cite[Lemma~4.1]{JZ18}}]\label{lem:modifying2form}
    Let $X$ be a $4$--manifold with connected boundary $Y$ and $b_2^+(X)>1$, and $N$ an admissible cut of $X$ separating $X$ into $W_1 \cup W_2$. Suppose $b$ is an element in $H^2(X,\partial X; \R)$. Then there is a closed 2-from $\omega$ on $(X,\partial{X})$ such that
    \begin{itemize}
        \item $\omega$ vanishes on $N$, and
        \item $[\omega] = b$. 
    \end{itemize}
\end{lemma}

We also have the following. 
\begin{lemma}[Juh\'{a}sz--Zemke \protect{\cite[Lemma~4.2]{JZ18}}]\label{lem:finitespinc}
    Let $X$ be a $4$--manifold with connected boundary $Y$ and $b_2^+(X)>1$, and $N$ an admissible cut of $X$ separating $X$ into $W_1 \cup W_2$. If $\omega$ is a closed $2$-form on $X$ that is exact on $Y$ and $N$, then $F^-_{W_1,\frt;\omega|_{W_1}}$ and $F^+_{W_2,\fru;\omega|_{W_2}}$ are non-vanishing for only finitely many $\frt\in\Spin^c(W_1)$ and $\fru\in\Spin^c(W_2)$.  
\end{lemma}
The argument is almost exactly the same as in \cite[Lemma~4.2]{JZ18}. More precisely, in \cite[Lemma~4.2]{JZ18}, $X$ was a closed manifold so $W_1$ was a cobordism from $N$ to $S^3$ and $W_2$ was a cobordism form $S^3$ to $N$. In our situation $W_2$ is the same sort of cobordism so the argument for $F^+_{W_2,\fru;\omega|_{W_2}}$ is the same as in \cite[Lemma~4.2]{JZ18}, but our $W_1$ is a cobordism from $N$ to $Y$. So we need to consider applying the map $F^-_{W_1,\frt;\omega|_{W_1}}$ to a finite set of generators of $HF^+(Y)$ instead of the unique generator of $HF^+(S^3)$, otherwise the argument is identical. 

Recall in the introduction we defined 
\[
\Phi_{X; \b} = \sum_{s\in\Spin^c(X)}  \Phi_{X,\frs} \cdot z_1^{\langle i_*(\frs-\frs_0)\cup b_1 ,[W,\partial W]\rangle} \cdots z_n^{\langle i_*(\frs-\frs_0)\cup b_n ,[W,\partial W]\rangle},
\]
which encodes much of the information in the Ozsv\'{a}th--Szab\'{o} invariant $\Phi_{X}$.
With this notation we can see how to compute the Ozsv\'{a}th--Szab\'{o} invariant using perturbed Heegaard Floer homology. 

\begin{proposition}[Juhasz--Zemke \protect{\cite[Proposition~4.3]{JZ18}}]\label{compute}
    Let $X$ be a $4$--manifold with connected boundary $Y$ and $b_2^+(X)>1$, and $N$ an admissible cut of $X$ separating $X$ into $W_1 \cup W_2$. Suppose $\b=\{b_1,\ldots,b_n\}$ is a basis of $H^2(X,\partial X;\R)$ and $\omegas=\{\omega_1,\ldots,\omega_n\}$ is a set of $2$-forms on $(X,\partial{X})$ such that $[\omega_i] = b_i$ and each $\omega_i$ vanishes on $N$. Then 
    \[
        \Phi_{X;\b} \doteq F^+_{W_2;\omegas|_{W_2}} \circ \tau^{-1} \circ F^-_{W_1;\omegas|_{W_1}}(\theta^-).
    \]
\end{proposition}

\section{Knot and concordance surgery}\label{KCS}



Concordance surgery is a generalization of Fintushel-Stern's knot surgery \cite{FS98}; see \cite{Akbulut02, Tange05}. 
Here we discuss the generalized version of  concordance surgery using self-concordance  in a homology sphere, introduced by Juh\'{a}sz and Zemke in \cite{JZ18}.

Let $X$ be a $4$--manifold with the same conditions containing a torus $T$ with trivial normal bundle and let $N_T=T\times D^2$ be a neighborhood of $T$ in $X$. Suppose $M$ is an integer homology $3$--sphere and $K$ is a knot in $M$. Consider a self-concordance  $\cC = (I \times M, A)$ from $(M,K)$ to itself. We may glue the ends of $I\times M$ and obtain $S^1\times M \cong I\times M/\sim$. We also glue the ends of $A$ and obtain an embedded torus $T_\cC\subset S^1 \times M$. Remove a neighborhood of $T_\cC$ from $S^1\times M$ to obtain a $4$--manifold $W_\cC$ with boundary $T^3$. Now we may glue $X\setminus N_T$ and $W_\cC$ so that $\partial D^2$ in $N_T$ is glued to the longitude for $K$ in $M\setminus N_K$. We write the resulting manifold as $X_\cC$ and call it a result of \dfn{concordance surgery on $X$}. Note that if we use a product concordance in $S^3$, then concordance surgery and knot surgery are same. 

Since $W_\cC$ is a homology $T^2\times D^2$ and the fact that we are gluing the null-homologous curve in $W_\cC$ to the meridian of $T$ we see that the homology and intersection pairing of $X$ and $X_\cC$ are the same. Moreover, if $\cC$ is a self-concordance of $(S^3,K)$ and $X\setminus N_T$ is simply connected, then so is $X_\cC$ and hence by Freedman \cite{Freedman82} $X$ and $X_\cC$ will be homeomorphic (this is because the fundamental group of the complement of $T_\cC$ in $S^1\times S^3$ is generated by meridians and $S^1\times \{pt\}$ and so will be killed when glued to $X\setminus N_T$). 

In \cite{JZ18}, Juh\'{a}sz and Zemke showed how the Ozsv\'{a}th--Szab\'{o} invariants change under concordance surgery. To state this, we recall the graded Lefschetz number of the concordance map on knot Floer homology.
Let $\cC = (I\times M, A)$ be a self-concordance of a knot $K$ in a homology $3$--sphere $M$ and let $a$ be a pair of properly embedded parallel arcs on $A$ connecting the boundary components of $A$. In \cite{Juhasz16}, Juh\'{a}sz showed that there is a map induced by $(\cC, a)$ on knot Floer homology:
\[
    \widehat{F}_{\cC,a}: \widehat{HFK}(M,K) \rightarrow \widehat{HFK}(M,K).
\]

In \cite{JM18}, Juh\'{a}sz and Marengon showed that the map $\widehat{F}_{\cC,a}$ preserves the Alexander and Maslov gradings when $M=S^3$ and in \cite{Zemke19}, Zemke extended this result to general $3$--manifolds. Thus, we can define the graded Lefschetz number of $\widehat{F}_{\cC,a}$:

\[
    \Lef_z(\cC) := \sum_{i\in\Z}\Lef\left(\widehat{F}_{\cC,a}|_{\widehat{HFK}(M,K,i)}:\widehat{HFK}(M,K,i) \rightarrow \widehat{HFK}(M,K,i) \right)\cdot z^i.
\]

In \cite{JZ18}, Juh\'{a}sz and Zemke proved that the graded Lefschetz number of $\widehat{F}_{\cC,a}$ is independent of the parallel arcs $a$.

\begin{lemma}[Juh\'{a}sz--Zemke \cite{JZ18}]\label{lem:symmetry}
The graded Lefschetz number    $\Lef_z(\cC)$ does not depend on the choice of arcs $a$. Moreover, $\Lef_z(\cC)$ is symmetric under the conjugation $z\mapsto z^{-1}$.
\end{lemma}

They also proved a concordance surgery formula for closed $4$--manifolds.
\begin{theorem}[Juh\'{a}sz--Zemke \cite{JZ18}]\label{consurg}
    Let $X$ be a smooth, oriented and closed $4$--manifold with $b_2^+(X) > 1$. Suppose $T$ is an embedded torus in $X$ with trivial normal bundle such that $[T]\neq0\in H_2(X;\R)$ and $\b=(b_1,\ldots,b_n)$ is a basis of $H^2(X;\R)$ such that $\langle[T],b_1\rangle=1$ and $\langle[T],b_i\rangle=0$ for $i>1$. Then
    \[
        \Phi_{X_\cC;\b} = \Lef_z(\cC)\cdot\Phi_{X;\b}
    \]        
\end{theorem}

A simple, but insightful, observation of Fintushel and Stern about knot surgery, which is naturally extended to concordance surgery, is the following.
\begin{lemma}[Fintushel and Stern 1998, \cite{FS98}]\label{lfs}
Let $(X,\omega)$ be a symplectic $4$--manifold and $T$ a symplectically embedded torus with trivial normal bundle. If $K$ is a fibered knot in a homology sphere $M$ and $\cC$ the trivial self-concordance of $(M,K)$, then $X_\cC$ may be constructed so that it has a symplectic structure $\omega_\cC$. Moreover, in the complement of the surgery region of $X_\cC$ and the neighborhood of the torus in $X$ the symplectic structures $\omega_\cC$ and $\omega$ agree.  
\end{lemma}

While the concordance surgery construction is very general, we will consider it in the case that $T$ sits nicely in a cusp neighborhood. A cusp neighborhood $C$ is the neighborhood of a cusp singular fiber in an elliptic fibration, see \cite{GS99} for more details. A handle picture for $C$ is given in Figure~\ref{nucleus}. 
\begin{figure}[ht]
\begin{overpic}
{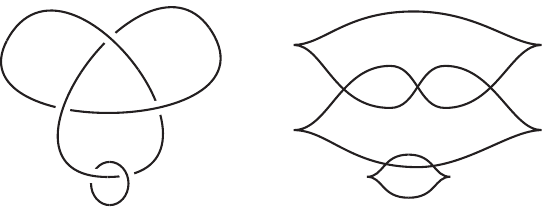}
\put(108, 58){$0$}
\put(63, 2){$-n$}
\end{overpic}
\caption{The surgery diagram on the left without the unknot is a cusp neighborhood. A regular torus fiber can be seen in the neighborhood by taking the punctured torus Seifert surface for the trefoil and capping it off with the core of the 2-handle. The diagram on the left is a Gompf nucleus $N_n$ which clearly contains a cusp neighborhood. On the right is a Weinstein diagram for $N_2$.}
\label{nucleus}
\end{figure}  
The complement of the singular fiber in the cusp neighborhood is a $T^2$ fibration over a punctured disk. We will consider concordance surgery along tori that are fibers in this fibration. 

In our applications below we will actually want the tori along which we do concordance surgery to be slightly more constrained. Specifically we would like them to lie in Gompf nuclei $N_n$  \cite{Gompf91}. This is an enlargement of a cusp neighborhood and can be thought of as a neighborhood of a cusp fiber and a section in an elliptic fibration. See Figure~\ref{nucleus} for a picture of $N_n$.

\section{Concordance surgery formula}\label{concsurgery}

In this section, we will prove Theorem~\ref{thm:main1} and Corollary~\ref{cor:main2}. 

We begin with a model case. Let $W_0=T^2\times D^2$ and $\cC=(I\times M, A)$ be a self-concordance of a knot in a homology sphere. Set $W_\cC$ to be the result of concordance surgery on $W_0$ along the torus $T^2\times \{0\}$. 
We will use Proposition~\ref{compute} to compute the Ozsv\'{a}th--Szab\'{o} invariant  of $W_\cC$, and thus we begin by computing 
the cobordism map $F^-_{W_\cC;\omega_\cC}$. Let $W_0$ be a $4$--manifold diffeomorphic to $T^2\times D^2$ and $\omega_0\in\Omega^2(W_0)$ a closed $2$-form on $W_0$ which is Poincar\'{e} dual to $\{p\}\times D^2$. Let $\omega_\cC\in\Omega^2(W_\cC)$ be a closed $2$-form on $W_\cC$ which is Poincar\'{e} dual to $\{p\}\times\Sigma_K$ where $\Sigma_K$ is a Seifert surface of $K$ in $Y\setminus N_K$. Since $\Spin^c(W_0)\cong\Spin^c(W_\cC)\cong\Z$ we can list the $\Spin^c$ structures of $W_0$ and $W_\cC$ as follows. Let $\frt_k\in\Spin^c(W_0)$ be the $\Spin^c$ structure on $W_0$ such that $c_1(\frt_k)=2k\cdot PD[\{p\}\times D^2]$. Similarly, let $\frt'_k\in\Spin^c(W_\cC)$ be the $\Spin^c$ structure on $W_\cC$ such that $c_1(\frt'_k) = 2k\cdot PD[\{p\}\times \Sigma_K]$. 

In \cite{JZ18}, Juh\'{a}sz and Zemke computed $F^+_{W_\cC;\omega_{\cC}}$.

\begin{proposition}[Juh\'{a}sz--Zemke \protect{\cite[Corollary~5.5]{JZ18}}]\label{prop:plus-map}
    Consider $W_\cC$ as a cobordism from $-T^3$ to $S^3$. Then
    \[
        F^+_{W_\cC,\frt_0';\omega_{\cC}} \doteq \Lef_z(\cC) \cdot F^+_{W_0,\frt_0;\omega_0}.
    \]
    For $k \neq 0$, the cobordism map $F^+_{W_\cC,t_k';\omega_{\cC}}$ vanishes.
\end{proposition}

By duality, we can prove the minus version of the above proposition.

\begin{proposition}\label{prop:minus-map}
    Consider $W_{\cC}$ as a cobordism from $S^3$ to $T^3$. Then
    \[
        F^-_{W_\cC,\frt_0';\omega_{\cC}}(\theta^-) \doteq \Lef_z(\cC) \cdot F^-_{W_0,\frt_0;\omega_0}(\theta^-).
    \]
    For $k\neq0$, the cobordism map $F^-_{W_\cC,\frt_k';\omega_{\cC}}$ vanishes.
\end{proposition}

\begin{proof}
    Let $\tau = \omega_{\cC}|_{T^3}$ and $a$ an element in $HF^+(T^3;\tau)$. Then
    \begin{align*}
        \langle a, F^-_{W_\cC,\frt_0';\omega_\cC}(\theta^-) \rangle &= \langle F^+_{W_\cC,\frt_0';\omega_\cC}(a), \theta^- \rangle\\
        &= \Lef_z(\cC) \langle F^+_{W_0,\frt_0;\omega_0}(a), \theta^- \rangle\\
        &= \Lef_z(\cC) \langle a, F^-_{W_0,\frt_0;\omega_0}(\theta^-) \rangle\\
        &= \langle a, \Lef_{z^{-1}}(\cC) \cdot F^-_{W_0,\frt_0;\omega_0}(\theta^-) \rangle\\
        &= \langle a, \Lef_{z}(\cC) \cdot F^-_{W_0,\frt_0;\omega_0}(\theta^-) \rangle
    \end{align*}
The first and third equality follows from the duality, Theorem~\ref{thm:duality}. The second equality follows from Proposition~\ref{prop:plus-map}. The fourth equality follows from the anti-linearity of the pairing. The last equation follows from the fact that $\Lef_z(\cC)$ is symmetric under the conjuation $z\mapsto z^{-1}$ as noted in Lemma~\ref{lem:symmetry}. Now the desired formula follows from the nondegeneracy of the pairing given in Theorem~\ref{thm:pairing}. The same proof works for $k\neq0$.
\end{proof}

Now we can obtain the minus version of \cite[Corollary~5.6]{JZ18}. The proof is identical.

\begin{corollary}\label{cor:localsurgery}
    Suppose $\omegas = (\omega_1,\ldots,\omega_n)$ is a collection of closed $2$-forms on $(X,\partial X)$ satisfying
    \[
        \int_T \omega_1 =1 \quad\text{and}\;\; \int_T \omega_i =0 \quad\text{for $i>1$},
    \]
    and $\omegas' = (\omega'_1,\ldots,\omega'_n)$ is the collection of induced $2$-forms on $(X_\cC, \partial X_\cC)$ under the canonical isomorphism $H^2(X_\cC, \partial X_\cC;\R)\cong H^2(X, \partial X;\R)$. Then
    \[
        F^-_{X_\cC,\frt'_0;\omegas'|_{X_\cC}}\doteq\Lef_z(\cC)\cdot F^-_{X_0,\frt_0;\omegas|_{X_0}}.
    \]
    For other $\Spin^c$ structures, both cobordism maps vainish.
\end{corollary}


Now we are ready to prove Theorem~\ref{thm:main1} that says if $X_\cC$ is the result of $\cC$ concordance surgery on $X$ then the Ozsv\'ath-Szab\'o polynomial of $X_\cC$ is that of $X$ multiplied by the graded Lefschetz number of $\cC$.

\begin{proof}[Proof of Theorem~\ref{thm:main1}]
Given our $4$--manifold with boundary $X$ that contains a torus $T$ with trivial normal bundle that is non-trivial in homology, choose an admissible cut $N$ the breaks $X$ into two pieces $W_1$ and $W_2$ where $W_1$ contains $T$ and $W_2$ contains the boundary $\partial X$. 

Let $\omegas = (\omega_1,\ldots,\omega_n)$ be a collection of closed $2$-forms on $(X,\partial X)$ such that $\int_T \omega_1=1$ and $\int_T \omega_i=0$ for $i>1$. Moreover, suppose that $\b=(b_1,\ldots,b_n)$, where $b_i=[\omega_i]\in H^2(X,\partial X;\R)$, is a basis of $H^2(X,\partial X)$. By Lemma~\ref{lem:modifying2form}, we can assume that $\omega_i|_N = 0$. From Proposition~\ref{compute}, 
we have
\[
    \Phi_{X;\b} \doteq F^+_{W_2;\omegas|_{W_2}} \circ \tau^{-1} \circ F^-_{W_1;\omegas|_{W_1}}(\theta^-).
\]

We can decompose $W_1$ as  $W \cup W_0$ where $W_0$ is a neighborhood of $T$ and $W = W_1\setminus \textrm{int}(W_0)$. The composition law (Proposition~\ref{prop:composition}) gives
\[
    F^-_{W_1;\omegas|_{W_1}}\doteq F^-_{W;\omegas|_{W}}\circ F^-_{W_0;\omegas|_{W_0}}.
\]
Let $W_1'$ be $W \cup W_\cC$. Then $X_\cC = W_2\cup W_1'$ and 
\[
    \Phi_{X_\cC;\omegas'} \doteq  F^+_{W_2;\omegas'|_{W_2}} \circ \tau^{-1} \circ F^-_{W_1';\omegas'|_{W_1'}}(\theta^-).
\]
We may apply the composition law again and obtain
\[
    F^-_{W_1';\omegas'|_{W_1'}}\doteq F^-_{W;\omegas'|_{W}}\circ F^-_{W_\cC;\omegas'|_{W_\cC}} 
\]
Since the surgery procedure does not modify $W$ and $W_2$, $\omegas'|_{W} = \omegas|_{W}$ and $\omegas'|_{W_2} = \omegas|_{W_2}$. Now the theorem follows from Corollary~\ref{cor:localsurgery}. 
\end{proof}

Now we prove Corollary~\ref{cor:main2}, following the proof of Corollary~1.2 in \cite{JZ18}. Recall this corollary says that the result of two different concordance surgeries on a $4$--manifold $X$ with non-trivial Ozsv\'ath-Szab\'o polynomial are not diffeomorphic if the their graded Lefschetz numbers are different. 

\begin{proof}[Proof of Corollary~\ref{cor:main2}]
    We will show that $X_{\cC}$ and $X_{\cC'}$ are not diffeomorphic if $\Lef_z(\cC)\neq \Lef_z(\cC')$. 
    Suppose $\phi:X_\cC\rightarrow X_{\cC'}$ is a diffeomorphism. By the naturality of Heegaard Floer homology, $\phi|_Y$ induces an automorphism on $HF^+(Y)$; see \cite{JTZ12}. We denote this automorphism by  $\phi_*$. Then we have
    \[
        \Phi_{X_{\cC'},\frs} = \phi_*(\Phi_{X_\cC,\phi^*(\frs)})
    \] 
Notice that $\phi_*$ naturally extends to a $\F_2[\Z^n]$-module automorphism on $\F_2[\Z^n]\otimes_{\F_2} HF^+(Y)$. 
After identifying $H^2(X_\cC,\partial X_\cC)$ and $H^2(X_{\cC'},\partial X_{\cC'})$ with $\Z^n$ via the basis $\b$ the action of $\phi^*$ on cohomology can be thought of as an element $M(\phi^*)$ of $\GL_n(\Z)$. We write $e^{M(\phi^*)^t}$ for the endomorphism of $\F_2[\Z^n] \otimes_{\F_2} HF^+(Y)$ given by $e^{M(\phi^*)^t} \cdot \z^\a \otimes b = \z^{M(\phi^*)^t \cdot \a} \otimes b$, where we view $\a$ as a column vector and $b \in HF^+(Y)$. With this notation we have the following equalities: 
    \begin{align*}
        \Phi_{X_{\cC'};\b} &= \sum_{\frs\in\Spin^c(X)} \z^{\langle i_*(\frs-\frs_0)\cup \b,[X,\partial X]\rangle}\cdot\Phi_{X_{\cC'},\frs}\\
        &= \sum_{\frs\in\Spin^c(X)} \z^{\langle \phi^*i_*(\frs-\frs_0)\cup\phi^*(\b),[X_\cC,\partial X_\cC]\rangle}\cdot\phi_*\left(\Phi_{X_\cC,\phi^*(\frs)}\right)\\
        &= \sum_{\frs\in\Spin^c(X)} \z^{\langle i_*(\frs-\phi^*(\frs_0))\cup\phi^*(\b),[X_\cC,\partial X_\cC]\rangle}\cdot\phi_*\left(\Phi_{X_\cC,\frs}\right)\\
        &\doteq \sum_{\frs\in\Spin^c(X)} \z^{\langle i_*(\frs-\frs_0)\cup\phi^*(\b),[X_\cC,\partial X_\cC]\rangle}\cdot\phi_*\left(\Phi_{X_\cC,\frs}\right)\\
        &= e^{M(\phi^*)^t}\cdot\sum_{\frs\in\Spin^c(X)} \z^{\langle i_*(\frs-\frs_0)\cup \b,[X_\cC,\partial X_\cC]\rangle}\cdot\phi_*\left(\Phi_{X_\cC,\frs}\right)\\
        &= \phi_*(e^{M(\phi^*)^t}\cdot\Phi_{X_\cC;\b}).
    \end{align*}
The first equality follows from the definition. The second equality follows from the naturality of the Heegaard Floer maps and of cohomology. The third equality follows from rearranging the sum and noting that as the sum goes over all of $\Spin^c(X)$ the $\phi^*(\frs)$ also ranges over all of $\Spin^c(X)$. The fourth equality follows since $\Phi_{X;\b}$ is invariant up to overall multiplication by a monomial and of the choice of base $\Spin^c$ structure $\frs_0$. The fifth equality is essentially the definition of $e^{M(\phi^*)^t}$.  The last equality follows from the definition of $\Phi_{X_\cC;\b}$ and the linearity of $\phi_*$.

The invariant $\Phi_{X_{\cC};\b}$ as an element of $\F_2[\Z^n]\otimes_{\F_2} HF^+(Y)$ can be written 
\[
\sum_{i=1}^k f_i\otimes x_i
\]
where $f_i\in \F_2[\Z^n]$ is a polynomial in $n$ variables and $x_i\in HF^+(Y)$. So we see that $\Phi_{X_{\cC'};\b}$ has the form
\[
\sum_{i=1}^k (e^{M(\phi^*)^t}\cdot f_i)\otimes \phi_*(x_i).
\]
Let $f_\cC$ be the greatest common divisor of $\{f_1, \ldots, f_k\}$. 
We claim that $f_{\cC'}=e^{M(\phi^*)^t} \cdot f_\cC$ is the greatest common divisor of the $\F_2[\Z^n]$ terms in $\Phi_{X_{\cC'};\b}$. Since $e^{M(\phi^*)^t}\cdot (fg) = (e^{M(\phi^*)^t}\cdot f)(e^{M(\phi^*)^t}\cdot g)$, clearly $f_{\cC'}$ is a common divisor of $\{e^{M(\phi^*)^t} \cdot f_1, ..., e^{M(\phi^*)^t} \cdot f_k\}$. Since $\phi_*$ is an automorphism of $\F_2[\Z^n] \otimes_{\F_2} HF^+(Y)$, $e^{M(\phi^*)^t}$ is also an automorphism of $\F_2[\Z^n]$. Thus $f_{\cC'}$ is the greatest common divisor of $\{e^{M(\phi^*)^t} \cdot f_1, ..., e^{M(\phi^*)^t} \cdot f_k\}$. 

Let $\alpha\in \F_2[\Z^n]$ be an irreducible element and let $f\in \F_2[\Z^n]$ be any element and $\psi$ an automorphism of $\F_2[\Z^n]$. 
In \cite{Sunukjian15}, Sunukjian defined the invariant $\Gamma_{\alpha, \psi}(f)$ to be the  the number of elements of the form $\psi^n(\alpha)$ and $\psi^n(\overline \alpha)$ that can be factored out of $f$, where $\overline \alpha$ just negates the $\Z^n$ elements of $\F_2[\Z^n]$.

In \cite{Sunukjian15}, Sunukjian showed that $\Gamma_{\alpha, \psi}(f)=\Gamma_{\alpha, \psi}(f)\circ \psi$  and that $\Gamma_{\alpha, \psi}(fg)=\Gamma_{\alpha, \psi}(f)+\Gamma_{\alpha, \psi}(g)$. He also showed that if $\Delta_K(z)\not=\Delta_{K'}(z)$, then there is some $\alpha\in \F_2[\Z^n]$ such that $\Gamma_{\alpha, \psi}(\Delta_K(z))>\Gamma_{\alpha, \psi}(\Delta_{K'}(z))$. The proof only used the fact that the Alexander polynomial is symmetric under sending $z$ to $z^{-1}$. But as noted in Lemma~\ref{lem:symmetry} the graded Lefschetz number $\Lef_z(\cC)$ has this same symmetry. So there is some $\alpha\in \F_2[\Z^n]$ such that $\Gamma_{\alpha, \psi}(\Lef_z(\cC))>\Gamma_{\alpha, \psi}(\Lef_z(\cC'))$.


As noted above $f_{\cC'}=e^{M(\phi^*)^t} \cdot f_\cC$ so we have that $\Gamma_{\alpha,\phi^*}(f_\cC)=\Gamma_{\alpha, \phi^*}(f_{\cC'})$ (since $e^{M(\phi^*)^t}$ is the action of $\phi^*$ on cohomology in the basis above). However, Theorem~\ref{thm:main1} says that  $\Phi_{X_\cC;\b} =\Lef_{z_1}(\cC) \cdot \Phi_{X;\b}$ and $\Phi_{X_{\cC'};\b} = \Lef_{z_1}(\cC) \cdot \Phi_{X;\b}.$ Thus if $f$ is the greatest common divisor of the $\F_2[\Z^n]$ terms in $\Phi_{X;\b}$, then $f_\cC=\Lef_{z_1}(\cC)  f$ and $f_{\cC'}=\Lef_{z_1}(\cC') f$ and we see that 
\[
\Gamma_{\alpha, \phi^*}(f)+ \Gamma_{\alpha, \phi^*}(\Lef_{z_1}(\cC))=\Gamma_{\alpha, \phi^*}(f_\cC)=\Gamma_{\alpha, \phi^*}(f_{\cC'})=\Gamma_{\alpha, \phi^*}(f)+ \Gamma_{\alpha, \phi^*}(\Lef_{z_1}(\cC')).
\]
So we see that $ \Gamma_{\alpha, \phi^*}(\Lef_{z_1}(\cC))= \Gamma_{\alpha, \phi^*}(\Lef_{z_1}(\cC'))$, contradicting the choice of $\alpha$ above. Hence the diffeomorphism $\phi$ cannot exist. 
\end{proof}

\section{Symplectic caps}\label{caps}

We begin by proving Theorem~\ref{cap} that says any closed contact $3$--manifold $(Y,\xi)$ has a (strong) symplectic cap $(X,\omega)$ that is simply connected and contains a Gompf nucleus $N_2$ whose regular fiber is symplectic and has simply connected complement. Moreover, $\Phi_{X,\frs_0}=c^+(\xi)\in H^+(-Y,\frs_0|_Y)$ for the canonical $\Spin^c$ structure $\frs_0$ for $(X,\omega)$.  

\begin{proof}[Proof of Theorem~\ref{cap}]
    We will build the strong symplectic cap $(X,\omega)$ in four steps.
    
    \noindent{\bf Step 1.} {\em Construct a simply connected cobordism $(X_1,\omega_1)$ from $(Y,\xi)$ to another contact manifold $(Y',\xi')$, that contains $N_2$.}

    The cobordism $(X_1,\omega_1)$ is built by adding Weinstein 2--handles to the convex end of the trivial symplectic cobordism $([0,1]\times Y, d(e^t\alpha))$, where $\alpha$ is a contact form for $\xi$ and $t$ is the coordinate on $[0,1]$. We begin by noting that one may attach a sequence of 2--handles to $[0,1]\times Y$ to kill the fundamental group as each 2--handle adds a relation to the fundamental group. The attaching circles of the 2--handles may be made Legendrian and the framings can be taken to be one less that the contact framings, thus we can take the handle attachments to be Weinstein 2--handle attachments. We finally attach two more Weinstein 2--handles as shown on the right hand side of Figure~\ref{nucleus}. The resulting cobordism is $(X_1,\omega_1)$.
    
    \noindent{\bf Step 2.} {\em Construct a cobordism $(X_2,\omega_2)$ consisting of Weinstein $2$--handle attachments from $(Y',\xi')$ to the contact manifold $(Y'',\xi'')$ where $Y''$ is a homology sphere.}

    A more detailed version of this argument may be found in \cite{Etnyre06}, but we sketch it here for the reader's convenience. 
    Let $(\Sigma,\phi)$ be an open book supporting the contact structure $(Y',\xi')$. By stabilizing the open book we can assume that $\Sigma$ has a single boundary component. Let $g$ be the genus of $\Sigma$. It is well-known that the mapping class group of $\Sigma$ is generated by Dehn twists about $\alpha_1,\ldots, \alpha_{2g+1}$ shown in Figure~\ref{surface}. All facts we use about diffeomorphisms of surfaces are well-known and can be found, for example, in \cite{FM12}.
    \begin{figure}[ht]
    \begin{overpic}
    {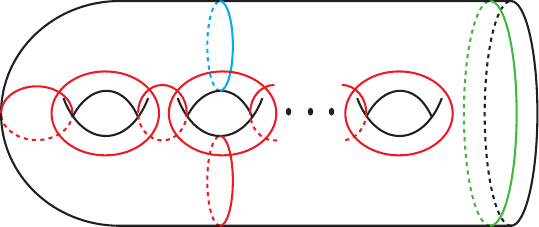}
    \put(10, 73){$\alpha_1$}
    \put(73, 73){$\alpha_3$}
    \put(55, 53){$\alpha_2$}
    \put(123, 33){$\alpha_4$}
    \put(115, 92){$\beta$}
    \put(115, 14){$\alpha_{2g+1}$}
    \put(184, 80){$\alpha_{2g}$}
    \put(240, 54){$\gamma$}
    \end{overpic}
    \caption{The surface $\Sigma$.}
    \label{surface}
    \end{figure}
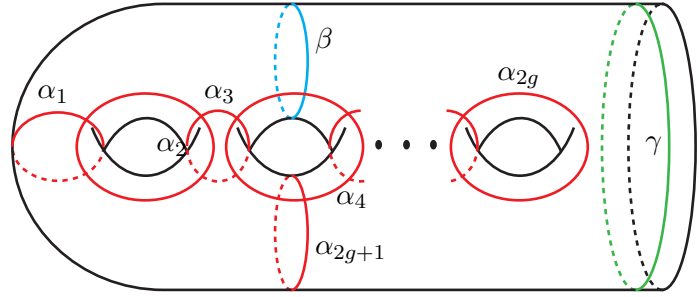  
    Fix some factorization of $\phi$ in terms of these Dehn twists. We will begin Step 2 by attaching $2$--handles so that the upper boundary has monodromy that is a composition of Dehn twist about only the curves $\alpha_1,\ldots, \alpha_{2g}$. We may do this as follows. If there are any negative Dehn twists about $\alpha_{2g+1}$ in the factorization then after conjugating $\phi$ (which does not change the contact manifold supported by the open book) we can assume it is at the end of the factorization, then we can attach a Weinstein 2--handle to $\alpha_{2g+1}$ as in Theorem~\ref{thm:2h}. This gives a cobordism where the upper boundary had the right handed Dehn twist about $\alpha_{2g+1}$ added, and this cancels the left handed Dehn twist. So we can now assume there are only right handed Dehn twists about $\alpha_{2g+1}$. If there is one, conjugate it to the end of the factorization and add a Weinstein 2--handle to $\beta$. Now apply the chain relation 
    \[
    \tau_{\alpha_{2g+1}}\tau_\beta= (\tau_{\alpha_1}\tau_{\alpha_2}\tau_{\alpha_3})^4
    \]
    to remove the right handed Dehn twist about $\alpha_{2g+1}$ (and about $\beta$). We now have a symplectic cobordism with upper boundary a composition of Dehn twists about only the curves $\alpha_1,\ldots, \alpha_{2g}$. By attaching Weinstein 2--handles as in Theorem~\ref{thm:2h} to the curves $\alpha_1,\ldots, \alpha_{2g}$ as necessary we can arrange that the upper boundary is supported by an open book with factorization a power of $(\tau_{\alpha_1}\cdots \tau_{\alpha_{2g}})^{4g+2}$. Then applying the chain relation 
    \[
    (\tau_{\alpha_1}\cdots \tau_{\alpha_{2g}})^{4g+2}=\tau_\gamma
    \]
    we have a symplectic cobordism with upper boundary having monodromy $\tau_\gamma^n$ for some $n$. Now attach a Weinstein $2$--handle to each of $\alpha_1,\ldots, \alpha_n$. The upper boundary of this cobordism now has monodromy $\tau_{\alpha_1}\cdots \tau_{\alpha_{2g}}\tau_\gamma^n$. It is easy to check that  the open book with monodromy $\tau_{\alpha_1}\cdots \tau_{\alpha_{2g}}$ defines a knot in $S^3$ and adding $\tau_\gamma^n$ to the monodromy corresponds to doing $-1/n$ surgery on this knot. Thus the upper boundary of the cobordism is a homology sphere. 


    \noindent{\bf Step 3.} {\em Construct a strong symplectic cap $(X_3,\omega_3)$ for $(Y'',\xi'')$ such that $b_2^+(X_3) \geq 2$.}

 The construction of the cap is due to  Eliashberg \cite{Eliashberg04}. 
 Here we just briefly review the construction. Consider an open book $(\Sigma', \phi')$ for $(Y'', \xi'')$ from Step 2. Eliashberg \cite{Eliashberg04} showed that by attaching a $2$-handle to the binding of the open book with framing coming from the page, we obtain a symplectic cobordism $(V_1,\omega')$ from $Y''$ to $Y'''$ where $Y'''$ is a $\widehat\Sigma'$-bundle over $S^1$, where $\widehat\Sigma'$ is $\Sigma'$ with a disk capping off its boundary. Since the monodromy of the $Y'''$ is an identity, and $(\tau_{\alpha_1}\cdots \tau_{\alpha_{2g}})^{4g+2}$ is isotopic to the identity, there is a Lefschetz fibration $V_2$ over $D^2$ with genus $g$ fibers and vanishing cycles defining the monodromy $(\tau_{\alpha_1}\cdots \tau_{\alpha_{2g}})^{4g+2}$ that has $Y'''$ as its boundary. One may easily verify that $V$ is simply connected and using Theorem~2 in \cite{Ozbagci02} one can compute that $b_2^+(V_2)=2g^2\geq 2$. Clearly $V_2$ admits a symplectic structure $\omega''$. Now Eliashberg \cite{Eliashberg04} showed how to glue $(V_1,\omega')$ and $(V_2,\omega'')$ to get a symplectic cap $X_3''$ for $(Y'', \xi'')$.

    \noindent{\bf Step 4.} {\em Construct the cap $(X,\omega)$ for $(Y,\xi)$ with all the desired properties. }

Using Lemma~\ref{lem:glue} we may glue the three cobordisms constructed above together to get a cap $(X,\omega)$ for $(Y,\xi)$. Since the first cobordism is simply connected and the second two cobordisms are constructed with 2, 3, and $4$--handles we see that $X$ is simply connected. It clearly contains $N_2$ and a regular fiber $T$ in $N_2$ transversely intersects an $S^2$ (the $-2$--framed unknot in Figure~\ref{nucleus}), so we see that the meridian to $T$ is null-homotopic in the complement of $T$. Since the fundamental group of $X-T$ is generated by meridians, we see that $X-T$ is simply connected. Moreover, it is well-known that the regular fibers in the cusp neighborhood in Figure~\ref{nucleus} can be taken to be symplectic. This can be seen in several ways; one way is to take the Lagrangian punctured torus that the trefoil in Figure~\ref{nucleus} bounds and capping it with the Lagrangian core of the 2-handles. This gives a Lagrangian torus isotopic to a fiber. Since this fiber is not null-homologus (since it intersects the $S^2$ discussed above) the symplectic form may be perturbed to make the torus symplectic, {\em cf.\ }\cite{Etnyre98}. 
    
Finally, we show that $\Phi_{X,\frs_0}=c^+(\xi)$. Choose an admissible cut $N$ for both $X$ and $X_3$. Then we can decompose $X$ into 
    \[
        X = X_1 \cup X_2 \cup X_3' \cup_N X_3''.
    \] 

We will denote the restriction of $\frs_0$, the canonical $\Spin^c$ structure of $(X,\omega)$,  to any of the cobordisms above by just $\frs_0$. By Plamenevskaya \cite[Lemma~1]{Plamenevskaya04} and the fact that $Y''$ is a homology sphere, the Ozsv\'{a}th--Szab\'{o} invariant of $(X_3,\omega_3)$ on $\frs_0$ is the contact invariant of $(Y'',\xi'')$. Now we obtain
    \begin{align*}
        c^+(\xi)&= F^+_{X_1\cup X_2,\frs_0}(c^+(\xi''))\\
                &= F^+_{X_1\cup X_2,\frs_0} (\Phi_{X_3,\frs_0})\\
                &= F^+_{X_1\cup X_2,\frs_0} \circ F^+_{X_3',\frs_0}\circ\tau^{-1}\circ F^-_{X_3'',\frs_0}(\theta^-)\\
                &= F^+_{X_1\cup X_2\cup X_3',\frs_0}\circ\tau^{-1}\circ F^-_{X_3'',\frs_0}(\theta^-)\\
                &=\Phi_{X,\frs_0}
    \end{align*}
The first equality follows from the naturality of contact invariant under Stein cobordisms \cite{OzsvathSzabo05a}. The second equality follows from \cite{Plamenevskaya04} as we discussed above. Since $Y''$ is a homology sphere, $\frs_0$ restricted to $X_1\cup X_2\cup X_3'$ uniquely decomposes into $\Spin^c$ structures on $X_1\cup X_2$ and $X_3'$ respectively. Thus the fourth equality follows from the composition law (Proposition~\ref{prop:composition}). The third and fifth equalities follow from the definition.
    
Thus $(X,\omega)$ has all the desired properties. 
\end{proof} 

We are now ready to prove Theorem~\ref{thm:main3}. 

\begin{proof}[Proof of Theorem~\ref{thm:main3}]
Let $K_i$ be any sequence of knots with distinct Alexander polynomials, for example we could take the $K_i$ to be the $(2, 2i+1)$-torus knots. 
Let $X_i$ be the result of $K_i$ knot surgery on $X$ using the fiber torus in the cusp neighborhood guaranteed by Proposition~\ref{cap}. Since any torus knot is fibered, $X_i$ has a symplectic structure according to Lemma~\ref{lfs}. 

We are left to see that the $X_i$ are all non-diffeomorphic but are homeomorphic. For the latter, it is not hard to see an infinite family of the $X_i$ are homeomorphic using the following result of Boyer. 
\begin{theorem}[Boyer 1986, \cite{Boyer86}]\label{boyer}
Let $Y$ be a closed, oriented, connected $3$--manifold and $L$ be a symmetric pairing on $\Z^n$. There are finitely many homeomorphisms types of compact, simply connected, oriented $4$--manifolds with boundary $Y$ and intersection pairing isomorphic to $(\Z^n,L)$. 
\end{theorem}
However, as pointed out to the authors by Gompf, one can show the $X_i$ are all homeomorphic. To see this we first note that Boyer points out the corollary to his theorem that if $W$ is a $4$--manifold with even intersection form and homology sphere boundary, then any other $4$--manifold with the same intersection from and boundary is homeomorphic to $W$ (by a homeomorphism that might not be the identity on $\partial W$). Now consider the nucleus $N$ given in Figure~\ref{nucleus} with $n=2$ and used in our construction of the $X_i$. Let $N_i$ be the result of $K_i$ knot surgery on the torus in $N$. From above we see that all the $N_i$ are homeomorphic to each other, but we don't know the homeomorphisms are the identity on the boundary. To remedy this, we observe that the diffeomorphisms of $\partial N$ up to isotopy form a group of order $2$ and the non-trivial diffeomorphism extends over $N$, see \cite[Proof of Lemma~3.7]{Gompf91}. Since any homeomorphism of a $3$--manifold is isotopic to a diffeomorphism (by work of Cerf \cite{Cerf68} and Hatcher's proof of the Smale conjecture \cite{Hatcher83}) if the homeomorphism from $N$ to $N_i$ is not the identity on the boundary then we can compose with the non-trivial homeomorphism mentioned above to see that $N$ and $N_i$ are homeomorphic by a homeomorphism fixing the boundary. But now since $X_i$ is imply $(X\setminus N) \cup N_i$, we see that $X_i$ is homeomorphic to $X$ by the above constructed homeomorphism on $N$ and the identity on $X\setminus N$. 


Since $T$ is a symplectic torus we can choose $\omega$ so that
\[
    \int_T\omega=1.
\]
Thus we can choose $\b=(b_1,\ldots,b_n)$, a basis of $H^2(X,\partial X;\R)$ such that $\langle b_1,[T]\rangle=1$ and $\langle b_i,[T]\rangle=0$ for $i\geq1$. We use Proposition~\ref{cap} to see that $\Phi_{X;\b}\neq0$. Now we see the $X_i$ are pairwise not diffeomorphic by Corollary~\ref{cor:main2}.
\end{proof}

We now turn to the proof of Theorem~\ref{thm:main4} concerning exotic fillings of manifolds that are the boundary of weakly convex symplectic $4$--manifolds. 
\begin{proof}[Proof of Theorem~\ref{thm:main4}]
Since admitting a filling with exotic smooth structures is independent of the orientation we will assume that $Y=\partial X$ where $X$ is a $4$--manifold that admits a symplectic form $\omega$ with weakly convex boundary. 

An argument very close to that given in Step 2.\ of the proof of Proposition~\ref{cap}, allows us to attach $2$--handles to $(X,\omega)$ to obtain a new symplectic manifold $(X'',\omega'')$ with weakly convex homology sphere boundary into which $(X,\omega)$ embeds. For details see \cite{Etnyre04}. If $Y$ is a rational homology sphere, a weak symplectic filling can be modified near the boundary into a strong symplectic filling \cite{OO99}. Thus we may assume that $(X'',\omega'')$ has a strongly convex boundary and can hence be capped off as above yielding the claimed symplectic manifold $(X',\omega')$ into which $(X,\omega)$ embeds and for which $X'\setminus X$ is simply connected and contains a cusp neighborhood that contains a symplectic torus. 

Let $K_i$ be the $(p^i,q)$-torus knot used in the proof of Theorem~\ref{thm:main3}. Let $X'_i$ be the result of $K_i$ knot surgery on $X'$ using the fiber torus in the cusp neighborhood. According to Lemma~\ref{lfs}, $X'_i$ has a symplectic structure $\omega_i$. Clearly $(X,\omega)$ embeds in all of these and so $(C_i,\omega_{C_i})=(\overline{X'_i \setminus X}, \omega_i)$ are all symplectic caps for $(Y,\xi)$. 

We are left to see that the $C_i$ are all non-diffeomorphic but are homeomorphic. For the latter, notice that the $C_i$ are all obtained from $\overline{X' \setminus X}$ by $K_i$ knot surgery. As noted in Section~\ref{KCS} all the $C_i$ will have the same homology and intersection form; moreover, since the complement of the torus used for knot surgery is simply connected the $C_i$ are also simply connected. By the argument of Gompf above, we see that all of $C_i$ must be homeomorphic. 

To see all the $C_i$ are not diffeomorphic, we can use a result of Ozsv\'ath and Szab\'o \cite{OzsvathSzabo04} that says the Ozsv\'ath--Szab\'o invariant of a closed symplectic manifold is non-trivial. Thus $\Phi_{X;\b}\not=0$. Now Theorem~\ref{consurg} says 
\[
\Phi_{X'_i;\b}=\Delta_{K_i}(z) \cdot \Phi_{X;\b}.
\]
Recall that for the trivial self-concordance $\cC$ of a knot in $S^3$, $\Lef_z(\cC)$ is simply $\Delta_K(z)$. As argued in the proof of Corollary~\ref{cor:main2}, $\Phi_{X'_i;\b}$ will be distinct and $X_i'$ will be non-diffeomorphic. But if any two of the $C_i$ are diffeomorphic by a diffeomorphism that is the identity on the boundary, then we could extend the diffeomorphism over the corresponding $X'_i$. Thus the $C_i$ are not diffeomorphic. 



By performing knot surgery on $X'$ using knots with non-monic Alexander polynomials, we can similarly construct an infinite number of smooth structures that do not admit symplectic structures. The corresponding $\overline{X' \setminus X}$ cannot admit symplectic structures giving a cap for $(Y,\xi)$ or we could glue them to $(X,\omega)$ to obtain a symplectic structure for the smooth structure on $X'$.  
\end{proof}

\section{Embeddings in definite manifolds}\label{df}
In this section we will prove Theorem~\ref{thm:main5} that says, if a rational homology sphere embeds as a separating hypersurface in a closed definite manifold, then it is the boundary of a smooth $4$--manifold that admits infinitely many distinct smooth structures. 
\begin{proof}[Proof of Theorem~\ref{thm:main5}]
Suppose $Y$ is a closed, oriented, connected $3$--manifold that embeds as a separating hypersurface in the closed definite $4$--manifold $W$. By reversing the orientation on $W$ if necessary we can assume it is negative definite. Using a Morse function on $W$ that has $Y$ as a regular level set, we can construct a handlebody structure on $W$ where the $1$-handles are disjoint from $Y$ and thus if the first Betti number of $W$ was positive, we could surger circles disjoint from $Y$ corresponding to the appropriate $1$-handles to kill the first Betti number and not change the second homology. 

We may remove two balls from $W$ that are on opposite sides of $Y$ to get a cobordism $W^\circ$ from $S^3$ to itself. In the proof of Theorem~9.1 in \cite{OzsvathSzabo03} it is shown that for any $\Spin^c$ structure $\frs$ on $W^\circ$ we have 
\[
F^\infty_{W^\circ \!\!, \frs} : HF^\infty(S^3, \frt)\to HF^\infty(S^3, \frt)
\]
is an isomorphism where $\frt$ is the unique $\Spin^c$ structure on $S^3$ and thus
\[
F^+_{W^\circ \!\!,\frs}: HF^+(S^3, \frt)\to HF^+(S^3, \frt)
\]
is surjective. 

Now as we know the intersection form of $W$ is diagonalizable by Donaldson's theorem \cite{Donaldson83}, it is not hard to see that there is a characteristic element $v=(1,\dots,1)$ in $H^2(W)$ such that $v\cdot v + n=0$ where $n$ is the dimension of $H^2(W)$. Since $v$ is characteristic we know there is some $\Spin^c$ structure $\frs_0$ whose first Chern class is $v$. 

From our discussion above if $\Theta^+$ is a generator of $HF^+(S^3, \frt)$ of minimal grading, then there is some element $\eta\in HF^+(S^3, \frt)$ that is mapped to it by $F^+_{W^\circ \!\!,\frs_0}$. The grading shift formula says 
\[
gr(F^+_{W^\circ \!\!,\frs_0}(\eta))- gr(\eta)=\frac{c_1^2(\frs_0) - 2\chi(W) - 3\sigma(W)}{4}= \frac{c_1^2(\frs_0) + n}{4}=0.
\]
and thus $\eta=\Theta^+$.

Notice that since $Y$ is a rational homology sphere, a $\Spin^c$ structure on $W^\circ$ is uniquely determined by its restrictions to $W_1$ and $W_2$, where $W_1$ and $W_2$ are the components of $W^\circ \setminus Y$. From this we know that if $\frs_0$ also denotes the restriction of $\frs_0$ to $W_1$, then we see that $F^+_{W_1,\frs_0}(\Theta^+)\not= 0$ in $HF^+(Y,\frs_0|Y)$.  

Now let $X$ be the result of gluing the $S^3$ boundary component of $W_1$ to the $K3$-surface minus a $3$--ball. Since the $K3$-surface is symplectic we know that its Ozsv\'ath--Szab\'o invariant is non-zero. In particular, the mixed map of the $K3$-surface sends the top generator $\Theta^-$ in $HF^-(S^3, \frt_0)$ to $\Theta^+$ in $HF^+(S^3,\frs_0)$. From this we know that $X$ (minus a ball) thought of as cobordism from $S^3$ to $Y$ has a $\Spin^c$ structure $\frs_0$ such that $F^{mix}_{X, \frs_0}(\Theta^-)\not=0$ in $HF^+(Y,\frs_0|_Y)$. Thus 
the Ozsv\'ath--Szab\'o invariant $\Phi_{X;\b}$ is non-trivial. 

Recall the $K3$-surface contains a cusp neighborhood with symplectic regular fibers and thus $X$ does too and we can apply Corollary~\ref{cor:main2} to obtain infinitely many non-diffeomorphic smooth structures on $X$. We notice that all of these structures are homeomorphic since the knot surgeries are done in the $K3$-surface and hence to not change the topological type of it and when glued to the fixed manifold $W_1$ we still have homeomorphic manifolds. 
\end{proof}

\bibliography{references}{}
\bibliographystyle{plain}

\end{document}